\documentclass{amsart}

\usepackage{amsmath,amsthm,amssymb}

\usepackage{graphicx}

\title{Free subgroups in groups acting on rooted trees}
\author{Volodymyr Nekrashevych}\thanks{This material is based upon work supported by the National Science Foundation under Grant DMS-0605019.}

\newcommand{\alb}{\mathsf{X}}
\newcommand{\xs}{\alb^*}
\newcommand{\xo}{\alb^\omega}
\newcommand{\nuke}{\mathcal{N}}
\newcommand{\arr}{\longrightarrow}
\newcommand{\aut}[1]{\mathop{\mathrm{Aut}}\left(#1\right)}

\newtheorem{theorem}{Theorem}[section]
\newtheorem{proposition}[theorem]{Proposition}
\newtheorem{corollary}[theorem]{Corollary}
\newtheorem{lemma}[theorem]{Lemma}

\theoremstyle{definition}
\newtheorem{defi}{Definition}[section]

\newtheorem{examp}{Example}
\newtheorem{examps}[examp]{Examples}

\begin{document}
\begin{abstract}
We show that if a group $G$ acting faithfully on a rooted tree $T$
has a free subgroup, then either there exists a point $w$ of the
boundary $\partial T$ and a free subgroup of $G$ with trivial
stabilizer of $w$, or there exists $w\in\partial T$ and a free
subgroup of $G$ fixing $w$ and acting faithfully on arbitrarily
small neighborhoods of $w$. This can be used to prove absence of
free subgroups for different known classes of groups. For
instance, we prove that iterated monodromy groups of expanding
coverings have no free subgroups and give another proof of a
theorem by S.~Sidki.
\end{abstract}

\maketitle

\section{Introduction}

It is well known that free groups are ubiquitous in the
automorphism group of an infinite rooted spherically homogeneous
tree, see for
instance~\cite{bhattach:freegr,abertvirag:dimension}, though
explicit examples (especially ones generated by finite automata)
were not so easy to construct,
see~\cite{al:free_en,sibr:gln,olij:ccc_en,glasnermozes,vorobets:alfree}.

On the other hand, many famous groups, which are defined by their
action on rooted trees do not have free subgroups. Absence of free
subgroups is proved in different ways. In some cases it follows
from torsion or sub-exponential growth (as in the Grigorchuk
groups~\cite{grigorchuk:80_en,grigorchuk:growth_en} and
Gupta-Sidki groups~\cite{gupta-sidkigroup}). In other cases it is
proved using some contraction arguments (see, for
instance~\cite{zukgrigorchuk:3st}).

S.~Sidki has proved in~\cite{sidki:polynonfree} absence of free
groups generated by ``automata of polynomial growth'', which
covers many examples.

An important class of groups acting on rooted trees are the
\emph{contracting self-similar groups}. They appear naturally as
iterated monodromy groups of expanding dynamical systems
(see~\cite{nek:book}). There are no known examples of contracting
self-similar groups with free subgroups and it was a folklore
conjecture that they do not exist.

The intuition behind this conjecture and the theorems mentioned
above is that there is no sufficient ``room'' for free subgroups.
The action of the elements of the groups in all examples are
concentrated on small portions of the tree and the graphs of the
action of the groups on the boundary of the rooted tree are also
small in some sense. For instance, the graphs of the action of
contracting groups have polynomial growth.

We prove in our paper the following theorem formalizing this
intuition.

\medskip
\noindent\textbf{Theorem~\ref{th:mainlemma}. }{\itshape Let $G$ be
a group acting faithfully on a locally finite rooted tree $T$.
Then one of the following holds.
\begin{enumerate}
\item $G$ has no free non-abelian subgroups,
\item there is a free non-abelian subgroup $F<G$ and a point $w\in\partial T$
such that the stabilizer $F_w$ is trivial,
\item there is a point $w\in\partial T$ such that the group of $G$-germs
$G_{(w)}$ has a free non-abelian subgroup.
\end{enumerate}}
\medskip

Here the group of $G$-germs $G_{(w)}$ is the quotient of the
stabilizer $G_w$ by the subgroup of automorphisms $g$ of the tree
$T$ acting trivially on a neighborhood $U_g\subset\partial T$ of
$w$.

Theorem~\ref{th:mainlemma} was inspired by a result of
M.~Ab{\'e}rt implying that if $F$ is a free group acting
faithfully and level transitively on a rooted tree, then there
exists a point of the boundary of the tree having trivial
stabilizer in $F$, see~\cite{abert:conjugatechains}.

Even though the theorem itself is not very complicated, it gives a
way to find simple proofs of absence of free subgroups in many
groups acting on rooted trees. One has to show that the graphs of
the action of the group on the boundary are so small that a free
action of a free subgroup is not possible, and then to analyze the
action of the elements of the group on neighborhoods of fixed
points in order to show that the groups of germs of the action are
also small and have no free subgroups.

In particular, we confirm the conjecture on contracting groups.

\medskip
\noindent\textbf{Theorem~\ref{th:contracting}. } {\itshape
Contracting groups have no free subgroups.}
\medskip

This theorem implies, for instance, that the iterated monodromy
groups of post-critically finite rational functions and other
expanding dynamical systems (see~\cite{nek:book,bgn}) have no free
subgroups. It is an interesting open question if all contracting
groups are amenable.

We also generalize (in Theorem~\ref{th:bounded}) the fact that
there are no free groups generated by bounded automorphisms of a
rooted tree. For the notion of bounded automorphisms see
Definition~\ref{def:bdd} of our paper and the
articles~\cite{sid:cycl,bondnek,bknv}. It is known that groups
generated by bounded automorphisms \emph{defined by finite
automata} have no free subgroups~\cite{sidki:polynonfree} and that
they are even amenable~\cite{bknv}. Absence of free subgroups for
the general case of bounded automorphisms is proved here for the
first time. It is not known if they are all amenable.

We also give a shorter proof of the theorem of
S.~Sidki~\cite{sidki:polynonfree} about automata of polynomial
growth (only for the case of finite alphabets).

A very intriguing open question now is to see if all the groups
covered by these theorems are amenable and to prove a theorem on
amenability of groups acting on rooted trees similar to
Theorem~\ref{th:mainlemma}. The first more or less general result
in this direction is the proof of amenability of groups generated
by bounded automata in~\cite{bknv}.

\section{Preliminaries on rooted trees}
Here we recall the basic notions related to rooted trees and fix
notation. The reader can find more on this
in~\cite{bass,sidki_monogr,handbook:branch,grineksu_en}.

A \emph{rooted tree} is a tree with a fixed vertex called the
\emph{root} of the tree. We consider only locally finite trees in
our paper, i.e., trees in which every vertex belongs to a finite
number of edges.

We say that a vertex $u$ of a tree $T$ is \emph{below} a vertex
$v$ if the path from the root of $T$ to $u$ goes through $v$. We
denote by $T_v$ the subtree of all vertices which are below $v$
together with $v$ serving as a root of $T_v$. See
Figure~\ref{fig:T}.

\begin{figure}
  \includegraphics{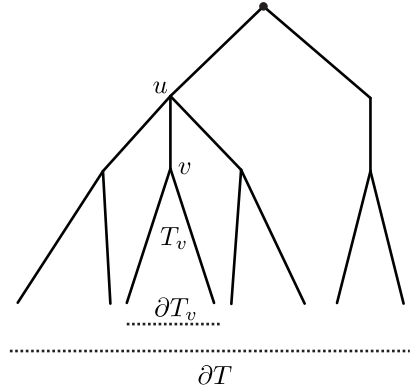}\\
  \caption{Rooted tree.}\label{fig:T}
\end{figure}

The \emph{boundary} $\partial T$ of the tree $T$ is the set of
simple infinite paths starting at the root of $T$. The boundary
$\partial T_v$ of a subtree $T_v$ consists then of the paths going
through the vertex $v$. The collection of subsets $\partial T_v$
of $\partial T$ is a basis of topology on $\partial T$. The
topological space $\partial T$ is compact and totally
disconnected. If $\tilde T$ is a subtree of $T$, then
$\partial\tilde T$ is a closed subset of $\partial T$ in the
natural way.

An automorphism of the rooted tree $T$ is an automorphism of the
tree $T$ fixing the root vertex. Every automorphism of $T$ fixes
also the \emph{levels} of the tree as sets. Here a \emph{level
number $n$} of the rooted tree $T$ is the set $L_n$ of the
vertices on distance $n$ from the root.

A rooted tree $T$ is said to be \emph{spherically homogeneous} if
the automorphism group of $T$ is transitive on the levels.

If $T$ is spherically homogeneous, then $\partial T$ is equipped
with a natural probability measure $m_T$ defined by the condition
that measure of $\partial T_v$ is equal to $1/|L_n|$, where $L_n$
is the level of the vertex $v$. This is the unique probability
measure invariant under the action of the automorphism group of
$T$.

\subsection{Trees of words and almost finitary
automorphisms}\label{ss:words} Let $X$ be a finite set, called
\emph{alphabet} and let $X^*$ be the free monoid generated by $X$,
i.e., the set of finite words $x_1x_2\ldots x_n$ over the alphabet
$X$, including the empty word $\varnothing$. The set $X^*$ has a
natural structure of a rooted tree, where a vertex $v\in X^*$ is
connected to the vertices of the form $vx$ and the empty word is
the root.

We denote by $v X^*$ the sub-tree of words starting by $v$, i.e.,
the sub-tree $X^*_v$ of vertices below the vertex $v$.

The boundary $\partial X^*$ is naturally homeomorphic to the set
of infinite sequences $X^\omega=\{x_1x_2\ldots\;:\;x_i\in X\}$
with the product topology (where $X$ is discrete).

More generally, if
\[\alb=(X_1, X_2, \ldots)\]
is a sequence of finite sets, then we denote
\[\xs=\bigcup_{n\ge 0}\alb^n,\] where $\alb^n=X_1\times
X_2\times\cdots\times X_n$ for $n\ge 1$ and
$\alb^0=\{\varnothing\}$. We denote by $|v|$ the length of a word
$v\in\xs$, i.e., number such that $v\in\alb^{|v|}$.

The set $\xs$ is a rooted tree with the root $\varnothing$ in
which an element $v\in\alb^n$ is connected to all elements of the
form $vx$ for $x\in X_{n+1}$.

The boundary of the tree $\xs$ is homeomorphic to the direct product
\[\xo=X_1\times X_2\times\ldots\]
of discrete sets.

The invariant measure on $\xo$ coincides in this case with the
\emph{uniform Bernoulli} measure defined as the direct product of
the uniform distributions on $X_n$.

We denote by $\alb_n$ the sequence
\[\alb_n=(X_{n+1}, X_{n+2}, \ldots)\]

If $g$ is an automorphsm of the tree $\xs$, then for every
$v\in\xs$ there exists a unique automorphism $g|_v$ of the tree
$\xs_{|v|}$ such that
\[g(vu)=g(v)g|_v(u)\]
for all $u\in\xs_{|v|}$.

It is easy to see that the following properties of this operation
hold
\begin{equation}
\label{eq:restriction} (g_1g_2)|_v=g_1|_{g_2(v)}g_2|_v,\qquad
\left(g|_v\right)^{-1}=g^{-1}|_{g(v)},\qquad
g|_{v_1v_2}=g|_{v_1}|_{v_2}.
\end{equation}

\subsection{Almost finitary automorphisms of $\xs$}
\begin{defi}
Let $g$ be an automorphism of the tree $\xs$. A sequence $w\in\xo$
is \emph{$g$-regular} if there exists a beginning $v\in\xs$ of $w$
such that $g|_v$ is trivial. We say that $w\in\xo$ is
\emph{$g$-singular} if it is not $g$-regular.
\end{defi}

The set of $g$-singular points of $\xo$ can be measured using the
\emph{growth} of the automorphism $g$.

\begin{defi}
\label{def:growthfunction} Let $g$ be an automorphism of the
spherically homogeneous rooted tree $\xs$. Then its \emph{growth
function} is
\[\theta_g(n)=|\{v\in\alb^n\;:\;g|_v\ne id\}|.\]

If $T\subset\xs$ is a \emph{rooted} subtree of $\xs$, i.e., a
sub-tree containing the root of $\xs$, then the \emph{relative
growth function} is
\[\theta_{g, T}(n)=|\{v\in L_n\;:\;g|_v\ne id\}|\]
and $g$ is \emph{almost finitary on $T$} if
\[\lim_{n\to\infty}\frac{\theta_{g, T}(n)}{|L_n|}=0,\]
where $L_n=T\cap\alb^n$ is the $n$th level of the tree $T$.
\end{defi}

\begin{proposition}
\label{pr:measureregular} Let $T$ be a spherically homogeneous
rooted subtree of $\xs$ and let $g$ be an automorphism of $\xs$.
Then
\[\lim_{n\to\infty}\frac{\theta_{g, T}(n)}{|L_n|}=m_T(\Sigma_g),\]
where $L_n$ is the $n$th level of the tree $T$ and $\Sigma_g$ is
the set of $g$-singular points $w\in\partial T\subset\xo$.
\end{proposition}

\begin{proof}
It is easy to see that the set of $g$-regular points $w\in\xo$ is
open, hence the set $\Sigma_g$ is closed in $\partial T$. It
follows from by definitions of $\theta_{g, T}$ and the measure
$m_T$ that the number $\frac{\theta_{g, T}(n)}{|L_n|}$ is equal to
$m_T(\Sigma_{g, n})$, where
\[\Sigma_{g, n}=\bigcup_{v\in L_n, g|_v\ne id}
v\xo\cap\partial T.\] The sequence of the sets $\Sigma_{g, n}$ for
$n=1,2,\ldots$ is decreasing and their intersection is $\Sigma_g$,
since $\Sigma_g$ is closed. This finishes the proof, by continuity
of the measure.
\end{proof}

Recall that if we have a group $G$ generated by a finite symmetric
set $S$ and acting on a set $M$, then the corresponding Schreier
graph is the graph with the set of vertices $M$ and the set of
edges $S\times M$, where an edge $(s, x)\in S\times M$ starts in
$x$ and ends in $s(x)$.

A locally finite graph is said to be \emph{amenable} if for every
$\epsilon$ there exists a finite set $F_\epsilon$ such that
\[\frac{|\partial F_\epsilon|}{|F_\epsilon|}<\epsilon,\]
where $\partial F_\epsilon$ is the set of edges beginning in
$F_\epsilon$ and ending outside of $F_\epsilon$. A regular
tree, in particular the Cayley graph of a free non-abelian group,
is an example of a \emph{non}-amenable graph.

\begin{proposition}
\label{pr:griamenable} Let $G$ be a finitely generated
automorphism group of $\xs$ and let $T\subset\xs$ be a
$G$-invariant rooted subtree on which $G$ acts level transitively.
If all elements of $G$ are almost finitary on $T$, then all
components of the Schreier graph of the action of $G$ on $\partial
T$ are amenable.
\end{proposition}

\begin{proof}
The proposition follows directly from the theorem
of~\cite{nek:griamenable} (which in turn is a corollary of a
result of V.~Kaimanovich~\cite{kaim:leaves}). But we prefer to
give here a direct proof constructing the sets $F_\epsilon$
(called \emph{F\o lner sets}).

Fix a finite symmetric generating set $S$ of $G$ and consider the
Schreier graph $\Gamma_n$ of the action of $G$ on the $n$th level
$L_n$ of the tree $T$.

Let $\Gamma_n'$ be the subgraph of $\Gamma_n$ which consists only
of the edges $(s, v)$ such that $s|_v$ is trivial. Note that if
$(s, v)$ belongs to $\Gamma_n'$, then the inverse edge $(s^{-1},
s(v))$ also belongs to $\Gamma_n'$. The number of edges in the
difference $\Gamma_n\setminus\Gamma_n'$ is equal to $\sum_{s\in
S}\theta_{s, T}(n)$.

Let $\Phi_1, \Phi_2, \ldots, \Phi_k$ be the connected components
of $\Gamma_n'$ (as sets of vertices). Then in $\Gamma_n$ we have
\[\frac{|\partial\Phi_1|+|\partial\Phi_2|+\cdots+|\partial\Phi_k|}{|\Phi_1|+|\Phi_2|+\cdots
+|\Phi_k|}\le\frac{\sum_{s\in S}\theta_{s, T}(n)}{|L_n|},\] hence
there exists a component $\Phi_i$ such that
\[\frac{|\partial\Phi_i|}{|\Phi_i|}\le\frac{\sum_{s\in S}\theta_{s, T}(n)}{|L_n|}.\]

Consider an orbit $O$ of the action of $G$ on $\partial T$. Since
the action of $G$ on $T$ is level-transtive, there exists $w\in O$
such that the beginning $v_1$ of length $n$ of $w$ belongs to
$\Phi_i$. Let $u\in\xo_n$ be such that $w=v_1u$. By the definition
of $\Phi_i$, for every $v\in\Phi_i$ there exists an element $g\in
G$ such that $g(v_1)=v$ and $g|_{v_1}=id$ (see the first two
properties of~\eqref{eq:restriction} in
Subsection~\ref{ss:words}). Then $g(v_1u)=vu$, i.e., the point
$vu$ belongs to the orbit $O$. Let
\[F_n=\{vu\;:\;v\in\Phi_i\}\subset O.\]
If the edge $(s, v)$ from $v\in\Phi_i$ to $s(v)$ belongs to
$\Gamma_n'$, then $s|_v=id$, hence $s(vu)$ belongs to $\Phi_i$.
Consequently, in the Schreier graph of the action of $G$ on $O$ we
have
\[\frac{|\partial F_n|}{|F_n|}\le\frac{|\partial\Phi_i|}{|\Phi_i|}\le
\frac{\sum_{s\in S}\theta_{s, T}(n)}{|L_n|}\to 0,\] as
$n\to\infty$. This means that $F_n$ are F\o lner sets, i.e., that
the Schreier graph of $O$ is amenable.
\end{proof}

\section{Main theorem}
We will use the following simple lemmata. Their proofs are known
(see, for instance Proposition 1 of~\cite{sidki:polynonfree} for
Lemma~\ref{lemma2}), but we provide them for completeness.

\begin{lemma}
\label{lemma1} Let $F$ be a free non-abelian group and let $H<F$
be a cyclic subgroup. Then there exists a free non-abelian
subgroup $\tilde F<F$ such that $\tilde F\cap H$ is trivial.
\end{lemma}

\begin{proof}
We can find a subgroup of $F$ freely generated by the generator
$h$ of $H$ and two other elements $g_1, g_2$ of $F$ (take, for
instance, any element $g$ which does not commute with $h$ and then
take the index 2 subgroup of the free group $\langle g, h\rangle$
generated by $h, g^2$ and $g^{-1}hg$). Then we can take $\tilde
F=\langle g_1, g_2\rangle$.
\end{proof}

\begin{lemma}
\label{lemma2} Let $F$ be a free non-abelian group and let
\[\phi:F\arr G=G_1\times G_2\times\cdots\times G_n\] be a
homomorphism into a finite direct product of groups. If every
composition $\phi_i$ of $\phi$ with the projection $G\arr G_i$ has
a non-trivial kernel, then $\phi$ has a non-trivial kernel.
\end{lemma}

\begin{proof}
It is sufficient to prove the lemma for $n=2$. The general
statement will follow then by induction. Let $r_1$ and $r_2$ be
non-trivial elements of the kernels of $\phi_1$ and $\phi_2$,
respectively. If $r_1$ and $r_2$ belong to one cyclic subgroup of
$F$, then there exist $m_1, m_2\in\mathbb{Z}$ such that
$r_1^{m_1}=r_2^{m_2}$ is a non-trivial element of the kernel of
$\phi$. If they do not belong to a common cyclic subgroup, then
they do not commute and $[r_1, r_2]$ is a non-trivial element of
the kernel of $\phi$.
\end{proof}

\begin{defi}
Let $G$ be a group acting on a topological space $\mathcal{X}$.
The \emph{group of $G$-germs} of a point $x\in\mathcal{X}$ is the
quotient of the stabilizer of $x$ in $G$ by the subgroup of
elements acting trivially on a neighborhood of $x$. We denote it
by $G_{(x)}$.
\end{defi}

Informally speaking, $G_{(x)}$ describes the action of the
stabilizer $G_x$ locally on neighborhoods of $x$.

See an application of groups of germs to growth of groups
in~\cite{ershler:growth}.

\begin{theorem}
\label{th:mainlemma} Let $G$ be a group acting faithfully on a locally finite
rooted tree $T$. Then one of the following holds.
\begin{enumerate}
\item $G$ has no free non-abelian subgroups,
\item there is a free non-abelian subgroup $F<G$ and a point $w\in\partial T$
such that the stabilizer $F_w$ is trivial,
\item there is a point $w\in\partial T$ such that the group of $G$-germs
$G_{(w)}$ has a free non-abelian subgroup.
\end{enumerate}
\end{theorem}

\begin{proof}
Suppose that on the contrary, $G$ has a free subgroup $F$, the
groups of $G$-germs $G_{(w)}$ for all $w\in\partial T$ have no
free subgroups and there is no free subgroup $\widetilde F$ and a
point $w\in\partial T$ such that the stabilizer $\widetilde{F}_w$
is trivial.

For every point $w\in\partial T$ the stabilizer $F_w$ contains a
free non-abelian subgroup, since otherwise $F_w$ is cyclic or
trivial, and then we can find by Lemma~\ref{lemma1} a free
non-abelian subgroup $\widetilde F<F$ such that
$\widetilde{F}_w=\widetilde{F}\cap F_w$ is trivial. But the group
$G_{(w)}$ has no free subgroups, consequently the natural
homomorphism $F_w\longrightarrow G_{(w)}$ has a non-trivial
kernel. This means that there exists a vertex $v_w$ on the path
$w$ and a non-trivial element of $F$ acting trivially on the
sub-tree $T_{v_w}$.

We get a covering of $\partial T$ by open subsets $\partial
T_{v_w}$. The boundary of the tree is compact, hence there exists
a finite sub-covering. This means that there exists a finite set
of vertices $V$ such that $\partial T=\bigcup_{v\in V}\partial
T_v$ and for every $v\in V$ the group of the elements of $F$
acting trivially on $T_v$ is non-trivial, hence infinite.

Denote by $F_V$ the intersection of the stabilizers of the vertices of $V$ in $F$. It
is a subgroup of finite index in $F$. Consider the homomorphism
\[\phi:F_V\arr\prod_{v\in V}\aut{T_v}\]
mapping an automorphism $g\in F_V$ to the automorphisms of the
trees $T_v$ that it induces. Composition of $\phi$ with the
projection onto $\aut{T_v}$ has non-trivial kernel for every $v\in
V$, since $F_V$ has finite index in $F$ and the group of elements
of $F$ acting trivially on $T_v$ is infinite for each $v\in V$.
Consequently, by Lemma~\ref{lemma2}, the kernel of the
homomorphism $\phi$ is non-trivial. But this is not possible,
since we assume that $F$ acts faithfully on $\partial T$ and
$\partial T=\bigcup_{v\in V}\partial T_v$.
\end{proof}

\begin{examp}
Let $F$ be a free group. Then there exists a descending series of
finite index subgroups $F=G_0>G_1>G_2>\ldots$ with trivial
intersection $\bigcap_{n\ge 0}G_n$. Every such a series defines an
action of $F$ on the \emph{coset tree}. It is the rooted tree with
the set of vertices equal to the union $\bigcup_{n\ge 0}F/G_n$ of
the sets of cosets. The set $F/G_n$ is the $n$th level of the tree
and a vertex $gG_n\in F/G_n$ is connected by an edge with a vertex
$hG_{n+1}\in F/G_{n+1}$ if and only if $hG_{n+1}\subset gG_n$. The
group $F$ acts then on the coset tree by the natural action on the
cosets:
\[g\cdot(hG_n)=(gh)G_n.\]
The action is level transitive and the stabilizer of the path $G_0, G_1,
G_2, \ldots$ is equal to the intersection of the subgroups $G_n$,
i.e., is trivial.

It is not hard to prove (see, for instance, Proposition~4.4
of~\cite{lavnek}) that every level transitive action of $F$ on a
rooted tree $T$ for which there exists a path $w\in\partial T$
with trivial stabilizer is conjugate to the action on a coset
tree.
\end{examp}

\begin{examp}
Let $T$ be a rooted tree and $w=(v_0, v_1, \ldots)\in\partial T$
be an infinite simple path starting in the root, where $v_i$ are
the vertices that it goes through. Then the stabilizer of $w$ in
the $\aut T$ is isomorphic to the Cartesian product of the
automorphism groups of the sub-trees $\tilde
T_{v_i}=T_{v_i}\setminus T_{v_{i+1}}$ ``hanging'' from the path
$w$. For any free group $F$ consider a sequence of homomorphisms
$\phi_i:F\arr\aut{\tilde T_{v_i}}$ such that the intersection of
kernels $\bigcap_{i\ge 0}\ker\phi_i$ is trivial. Then we get a
faithful action of $F$ on the tree $T$ for which $w$ is a fixed
point. If the intersection $\bigcap_{i\ge n}\ker\phi_i$ is trivial
for every $n$, then the image of $F$ in the group of germs
$\aut{T}_{(w)}$ is faithful.
\end{examp}

\section{Applications}
\subsection{Contracting groups}

Let $\alb=(X, X, \ldots)$ be a constant sequence.

\begin{defi}
A \emph{self-similar group} is a group $G<\aut{\xs}$ such that for
every $g\in G$ and $v\in\xs$ we have $g|_v\in G$.
\end{defi}

It is sufficient to check that $g|_x\in G$ for every $x\in X$ and
every generator $g$ of $G$, due to the
properties~\eqref{eq:restriction} in Subsection~\ref{ss:words}.

\begin{defi}
A self-similar group $G$ of automorphisms of $\xs$ is called
\emph{contracting} if there exists a finite set $\nuke\subset G$
such that for every $g\in G$ there exists $n$ such that
$g|_v\in\nuke$ for all words $v$ of length more than $n$. The
smallest set $\nuke$ satisfying this condition is called the
\emph{nucleus} of the contracting group.
\end{defi}

It is proved in~\cite{nek:book} Proposition~2.13.8 that if a
finitely generated group $G$ is contracting then the growth of the
components of the Schreier graph of the action of $G$ on $\xo$ is
polynomial.

\begin{proposition}
\label{pr:finitegerms} If $G$ is a contracting group then for
every $w\in\xo$ the group of germs $G_{(w)}$ is finite of
cardinality not greater than the size of the nucleus.
\end{proposition}

\begin{proof}
Let $A\subset G_w$ be a subset such that such that $|A|>|\nuke|$,
where $\nuke$ is the nucleus of the group. There is $n$ such that
$g|_v\in\nuke$ for all $g\in A$ and all $v\in\alb^n$. In
particular, there exist $g, h\in A$ such that $g|_u=h|_u$, where
$u$ is the beginning of the length $n$ of the word $w$. It means
that $g^{-1}h$ acts trivially on the subtree $u\xs$, i.e., that
the images of $g$ and $h$ in $G_{(w)}$ are equal. Consequently,
$G_{(w)}$ has no finite subsets of size more than $|\nuke|$, i.e.,
$|G_{(w)}|\le |\nuke|$.
\end{proof}

\begin{theorem}
\label{th:contracting} Contracting groups have no free subgroups.
\end{theorem}

\begin{proof}
We have to eliminate the possibilities (2) and (3) of
Theorem~\ref{th:mainlemma}. Note that it is sufficient to prove
this theorem for finitely generated contracting groups, since
every finitely generated subgroup $H=\langle S\rangle$ of a
contracting group is a subgroup of a finitely generated
contracting group. One has to take the group generated by all
elements of the form $s|_v$ for $s\in S$ and $v\in\xs$.

Condition (2) can not be true, since the orbits of the action of a
contracting group have polynomial growth.

Condition (3) is not possible by Proposition~\ref{pr:finitegerms}.
\end{proof}

\begin{examps}
Iterated monodromy groups (see definitions in~\cite{bgn,nek:book})
of expanding self-coverings of orbispaces, in particular, iterated
monodromy groups of post-critically finite rational functions, are
contracting hence have no free subgroups. In some cases (like for
the polynomial $z^2+i$, see~\cite{buxperez:imgi}) the iterated
monodromy groups have sub-exponential growth, which obviously
implies that they have no free subgroups. In some other cases
absence of free groups was proved separately using some
contraction arguments. See for instance a proof of absence of free
subgroups in the iterated monodromy group of $z^2-1$
in~\cite{zukgrigorchuk:3st}.
\end{examps}

For more on contracting groups and their properties see the
monograph~\cite{nek:book}, especially its sections~2.11 and~2.13.

\subsection{Groups generated by bounded automorphisms}

Let $\xs$ be a level-tran\-si\-tive rooted tree defined by a
sequence
\[\alb=(X_1, X_2, \ldots)\]
of finite sets.

We say that an automorphisms $g$ of $\xs$ is \emph{finitary} if
there exists $n$ such that $g|_v$ is trivial for all $v\in\alb^n$.
(Recall that $\alb^n=X_1\times\cdots\times X_n$.) The smallest $n$
with this property is called the \emph{depth} of $g$.

Note that the set of all automorphisms of depth at most $n$ is a
finite subgroup of $\aut{\xs}$. In particular, the set of all
finitary automorphisms of $\xs$ is a locally finite group.

\begin{defi}
\label{def:bdd} An automorphism $g$ of $\xs$ is \emph{bounded} if
there exists a finite set of sequences $W=\{w_1, w_2, \ldots,
w_m\}\subset\xo$ and a number $n$ such that if $v\in\xs$ is not a
beginning of any sequence $w_i$ then $g|_v$ is finitary of depth
at most $n$. The number $n$ is called the \emph{finitary depth} of
$g$ and the set $W$ is called the set of \emph{directions} of $g$.
\end{defi}

Informally, an automorphism $g\in\aut{\xs}$ is bounded if its
activity is concentrated in strips of bounded width around a
finite number of paths in $\xs$.

Note that if $g$ is bounded, then the set $\Sigma_g$ of
$g$-singular points is a subset of the set of directions of $g$
and the sequence $\theta_g(n)$ is bounded.

\begin{proposition}
If $g_1, g_2$ are bounded automorphisms of $\xs$ of finitary depth
$\le n$, then $g_1^{-1}$ and $g_1g_2$ are bounded of finitary
depth $\le n$.
\end{proposition}

In particular, the set of bounded automorphisms of $\xs$ is a
group. Note that this group is uncountable.

\begin{proof}
It is a direct corollary of the properties~\eqref{eq:restriction}
from Subsection~\ref{ss:words} and the fact that the set of
finitary automorphisms of depth $\le n$ is a group.
\end{proof}

\begin{theorem}
\label{th:bounded} The group of bounded automorphisms of a
spherically homogeneous tree of bounded degree of vertices does
not contain free non-abelian subgroups.
\end{theorem}

\begin{proof}
Since the degree of vertices of the tree $\xs$ is bounded, the
sequence $(|X_i|)_{i\ge 1}$ is bounded.

Suppose that the group generated by bounded automorphisms contains
a finitely generated free non-abelian group. Then by
Theorem~\ref{th:mainlemma} either there exists a free group $F$
generated by bounded automorphisms and having trivial stabilizer of a
sequence $w\in\xo$, or there exists a free group $F$ fixing a
point $w\in\xo$ such that $F$ acts faithfully on every
neighborhood $v\xo_{|v|}$ of $w$.

Suppose that the first case holds. Let $T$ be the sub-tree of
$\xs$ such that $w\in\partial T$ and the free group $F$ acts level
transitively on $T$. (The tree $T$ is the union of the $F$-orbits
of beginnings of $w$.) The boundary $\partial T$ has to be
infinite, otherwise the stabilizer of $w\in\partial T$ is of
finite index. But for every $g\in F$ the set of $g$-singular
points of $\partial T$ is finite, hence of zero measure.
Consequently, by Proposition~\ref{pr:measureregular}
and~\ref{pr:griamenable}, the components of the Schreier graph of
the action of $F$ on $\partial T$ are amenable, which contradicts
with triviality of the stabilizer of $w$.

Suppose now that there exists $w$ such that a free 2-generated
group $F=\langle g, h\rangle$ fixes $w$ and acts faithfully on
neighborhoods of $w$. Let $W_g$ and $W_h$ be the sets of
directions of $g$ and $h$, respectively. Then $w\in W_g\cap W_h$,
since otherwise there would exist a neighborhood $v\xo_{|v|}$ of
$w$ such that $g|_v$ or $h|_v$ is trivial.

There exists a beginning $u$ of $w$ such that all elements of
$W_g$ and $W_h$ except for $w$ have beginning of length $|u|$
different from $u$. Then $g|_u$ and $h|_u$ are bounded
automorphisms of $\xs_{|u|}$ with one direction $w'$, where
$w=uw'$. They generate a free group, since $F$ acts faithfully on
$u\xo_{|u|}$.

Let $w'=x_1x_2\ldots$. Let $m$ be a number greater than the
finitary depths of $g'=g|_u$ and $h'=h|_u$. Then $g'$ and $h'$ may
change in a word $x_1x_2\ldots x_ny_{n+1}y_{n+2}\ldots$, for
$y_{n+1}\ne x_{n+1}$, only the letters $y_{n+1}, y_{n+2}, \ldots,
y_{n+m}$ and can not change $y_{n+1}$ to $x_{n+1}$. Since the size
of the alphabets $X_i$ are uniformly bounded, this implies that
the automorphisms $g'$ and $h'$ have finite order, which is a
contradiction.
\end{proof}

\begin{examps}
It is proved in~\cite{bknv} that groups generated by
\emph{finite-state} (see Definition~\ref{def:finitestate}) bounded
automorphisms of a rooted tree $\xs$ are amenable, which implies
that they have no free subgroups. Some examples of
non-finite-state groups of bounded automorphisms of $\xs$ are
studied in~\cite{grigorchuk:growth_en} and~\cite{nek:ssfamilies}.
\end{examps}

\subsection{Theorem of S.~Sidki}
Let again $\alb=(X, X, \ldots)$ be a constant sequence of finite
alphabets.

\begin{defi}
\label{def:finitestate} An automorphism $g$ of the tree $\xs$ is
said to be \emph{finite-state} if the set
\[\{g|_v\;:\;v\in\xs\}\subset\aut{\xs}\]
is finite.
\end{defi}

The set of all finite state automorphisms of the tree $\xs$ is a
countable group. It is proved in~\cite{sid:cycl} that if $g$ is
finite state then the growth function $\theta_g(n)$ (as given in
Definition~\ref{def:growthfunction}) has either exponential or
polynomial growth. The set $P_d(\alb)$ of all finite state
automorphisms of $\xs$ for which $\theta_g(n)$ is bounded by a
polynomial of degree $d$ is a group. A version of the growth function
$\theta_g(n)$ was defined and studied by S.~Sidki
in~\cite{sid:cycl}. Later in~\cite{sidki:polynonfree} S.~Sidki
proved that the groups $P_d(\alb)$ have no free subgroups
(actually, his theorem is more general, since it also covers some
cases of infinite alphabet $X$). Let us show how his theorem (for
the case of finite alphabet) follows from
Theorem~\ref{th:mainlemma}.

The following inductive description of elements of $P_d(\alb)$ is
given in~\cite{sid:cycl}. We denote by $P_{-1}(\alb)$ the group of
finitary automorphisms of $\xs$, i.e., the automorphisms $g$ for
which the sequence $\theta_g(n)$ is eventually zero.

\begin{proposition}
\label{pr:polydescription} Let $d$ be a non-negative integer. A
finite-state automorphism $g$ of $\xs$ belongs to $P_d(\alb)$ if
and only if there exists a finite number of almost periodic words
$W=\{v_iu_i^\omega\}\subset\xo$ such that the restrictions
$g_{i, n}=g|_{v_iu_i^n}$ do not depend on $n$, and
$g|_v\in P_{d-1}(\alb)$ if $v$ is not a beginning of
any $v_iu_i^\omega$.
\end{proposition}

\begin{corollary}
\label{cor:polyamenablegraphs} Let $G<P_d(\alb)$ be finitely
generated. Then every component of the Schreier graph of the
action of $G$ on $\xo$ is amenable.
\end{corollary}

It seems that the Schreier graphs of subgroups of $P_d(\alb)$
acting on $\xo$ have sub-exponential growth. That would of course
imply their amenability. See, for instance the graphs considered
in~\cite{CFS,benjaminihoffman}, which are Schreier graphs of a
sub-group of $P_1(\{0, 1\})$ (see Example~\ref{ex:tullio}).

\begin{proof}
Let us show that for every $g\in P_d(\alb)$ the set of
$g$-singular points of $\xo$ is at most countable. We argue by
induction on $d$. If $g\in P_{-1}(\alb)$, i.e., if $g$ is
finitary, then the set of $g$-singular points is empty. Suppose
that for every $g\in P_{d-1}(\alb)$ the set of $g$-singular points
is at most countable. Take arbitrary $g\in P_d(\alb)$. Then by
Proposition~\ref{pr:polydescription}, there is a finite set
$W\subset\xo$ such that $g|_v\in P_{d-1}(\alb)$ for every
$v\in\xs$ which is not a beginning of an element of $W$. Then the
set $\Sigma_g$ of $g$-singular points of $\xo$ is a subset of
\[W\cup\bigcup_{\text{$v\in\xs$ is not a beginning of any $w\in W$}}v\Sigma_{g|_v}.\]
The sets $\Sigma_{g|_v}$ are at most countable by the inductive
hypothesis. Hence $\Sigma_g$ is at most countable.

Let $G$ be a group generated by a finite set $S\subset P_d(\alb)$.
Let $T$ be a $G$-invariant rooted subtree of $\xs$ on which $G$
acts level transitively. Then the measure space $(\partial T,
m_T)$ is either a finite set, or it is isomorphic to the standard
Lebesgue space. In the first case the Schreier graph of the action
of $G$ on $\partial T$ is finite. In the second case the
components of the Schreier graph of the action of $G$ on $\partial
T$ are amenable by Propositions~\ref{pr:measureregular}
and~\ref{pr:griamenable}, since the sets of $g$-singular points of
$\partial T$ for $g\in S$ are at most countable, hence have zero
measure.
\end{proof}

\begin{theorem}[S.~Sidki]
\label{th:sidki} The group $P_d(\alb)$ has no free subgroups.
\end{theorem}

\begin{proof}
Note that $P_{-1}(\alb)$ is locally finite, hence has no free
subgroups. The case of $P_0(\alb)$ is covered by
Theorem~\ref{th:bounded}.

Condition (2) of Theorem~\ref{th:mainlemma} cannot hold for
$P_d(\alb)$ due to Corollary~\ref{cor:polyamenablegraphs}.

Let us investigate now the groups of germs in $P_d(\alb)$. Assume
that we have proved that $P_{d-1}(\alb)$ has no free subgroups.
Suppose that elements $a$ and $b$ of $P_d(\alb)$ generate a free
subgroup of the group of $P_d(\alb)$-germs of $w\in\xo$. It
follows from Proposition~\ref{pr:polydescription} that $w$ can be
represented in the form $w=v(u^\infty)$ for some finite words $v$
and $u$ so that $a|_{vu}=a|_v$, $b|_{vu}=b|_v$ and $a|_v$ and
$b|_v$ do not move $u$. Moreover, $a|_{vu^nu'}, b|_{vu^nu'}\in
P_{d-1}(\alb)$ for $u'\in\xs$ which are not beginnings of the word
$u^\infty$. Since $a$ and $b$ generate a free group of germs,
their restrictions $a_1=a|_v$ and $b_1=b|_v$ onto $v\xs$ generate
a free group.

Since $a_1|_u=a_1$ and $b_1|_u=b_1$, restrictions of the action of
$a_1$ and $b_1$ onto $\xs\setminus u\xs$ generate a free group $F$
(everything is periodic along the path $uuu\ldots$). Let
$u=x_1x_2\ldots x_m$ and let $U\subset\xs$ be the set of words of
the form $x_1x_2\ldots x_ky_{k+1}$, where $y_{k+1}\ne x_{k+1}$ and
$0\le k\le m-1$, i.e., the set of vertices adjacent to the path
from the root to $u$. Denote by $\tilde F$ the intersection of the
stabilizers in $F$ of the elements of $U$. See Figure~\ref{fig:u},
where the elements of $U$ are circled.

\begin{figure}
  \includegraphics{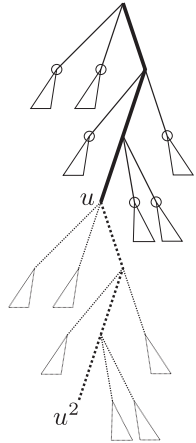}\\
  \caption{The set $U$ and the path $u^\infty$.}\label{fig:u}
\end{figure}

Then $\tilde F$ is a subgroup of finite index in $F$ and we get a
monomorphism
\[\phi:\tilde F\arr (\aut{\xs})^U\] mapping $g$ to $(g|_r)_{r\in
U}$. The homomorphism $\phi$ takes values in $(P_{d-1}(\alb))^U$,
hence it has a non-trivial kernel on each coordinate. Then, by
Lemma~\ref{lemma2}, the homomorphism $\phi$ has a non-trivial
kernel, which is a contradiction.
\end{proof}

\begin{examp}\label{ex:tullio}
Consider the permutations $a$ and $b$ of the set of integers given
by
\[a(n)=n+1\]
and
\[b(0)=0,\qquad b(2^k(2n+1))=2^k(2n+3)\]
for $k\ge 0$ and $n\in\mathbb{Z}$. Let $G$ be the group generated
by $a$ and $b$.

It is easy to see (using the binary numeration system on
$\mathbb{Z}$) that $G$ is isomorphic to the group of automorphisms
of the binary rooted tree $\{0, 1\}^*$ given by the recurrent
rules
\[a(0v)=1v,\qquad a(1v)=0a(v)\]
and
\[b(0v)=0b(v),\qquad b(1v)=1a(v).\]
A direct check (for instance using
Proposition~\ref{pr:polydescription}) shows that $a\in P_0(\{0,
1\})$ and $b\in P_1(\{0, 1\})$, hence the group $G$ has no free
subgroups by theorem of S.~Sidki.

The Schreier graphs of the action of $G$ on $\xo$ coincide with
the graphs considered by T.~Ceccherini-Silberstein, F.~Fiorenzi,
F.~Scarabotti~\cite{CFS} and I.~Benjamini,
C.~Hoffman~\cite{benjaminihoffman} and have intermediate growth.
\end{examp}

\providecommand{\bysame}{\leavevmode\hbox
to3em{\hrulefill}\thinspace}

\end{document}